\documentclass[12pt,a4paper]{article}
\usepackage{amsmath,amssymb,amsthm,graphicx,color}
\usepackage[left=1in,right=1in,top=1in,bottom=1in]{geometry}
\usepackage{cite}
\usepackage[british]{babel}
\usepackage{enumerate}
\usepackage{hyperref} 
\usepackage{bm}
\usepackage{comment}

\usepackage{tikz}
\usetikzlibrary{calc,arrows}

\hypersetup{
    colorlinks=false,
    pdfborder={0 0 0},
}

\newtheorem{theorem}{Theorem}[section]
\newtheorem{corollary}[theorem]{Corollary}
\newtheorem{proposition}[theorem]{Proposition}
\newtheorem{lemma}[theorem]{Lemma}
 
\newtheorem*{fact}{Fact}

\numberwithin{equation}{section}

\theoremstyle{definition}

\theoremstyle{remark}
\newtheorem{remark}[theorem]{Remark}
\newtheorem{remarks}[theorem]{Remarks}
\newtheorem*{remark*}{Remark}

 \allowdisplaybreaks

\newcommand{\1}[1]{{\mathbf 1}{\{#1\}}}
\newcommand{\2}[1]{{\mathbf 1}{(#1)}}

\newcommand{\R}{{\mathbb R}}
\newcommand{\Z}{{\mathbb Z}}
\newcommand{\N}{{\mathbb N}}
\newcommand{\ZP}{{\mathbb Z}_+}
\newcommand{\RP}{{\mathbb R}_+}
\newcommand{\Sp}[1]{{\mathbb S}^{#1}}

\newcommand{\X}{{\mathbb X}}

\DeclareMathOperator{\Exp}{\mathbb{E}}
\renewcommand{\Pr}{{\mathbb P}}

\DeclareMathOperator{\sign}{sgn} 
\DeclareMathOperator{\trace}{tr}

\def\diag#1{\mathop{\mathrm{diag}}\left( #1 \right)}

\newcommand{\tra}{{\scalebox{0.6}{$\top$}}}
\newcommand{\per}{{\mkern -1mu \scalebox{0.5}{$\perp$}}}

\newcommand{\eps}{\varepsilon}

\newcommand{\rc}{{\mathrm{c}}}

\newcommand{\cB}{{\mathcal B}}

\newcommand{\cF}{{\mathcal F}}
\newcommand{\cG}{{\mathcal G}}

\newcommand{\as}{\ \text{a.s.}}

\newcommand{\bigmid}{\; \bigl| \;}
\newcommand{\Bigmid}{\; \Bigl| \;}

\newcommand{\bx}{{\mathbf{x}}}
\newcommand{\by}{{\mathbf{y}}}
\newcommand{\bz}{{\mathbf{z}}}
\newcommand{\bu}{{\mathbf{u}}}

\newcommand{\be}{{\mathbf{e}}}
\newcommand{\0}{{\mathbf{0}}}
\newcommand{\bzeta}{{\bm{\zeta}}}

\newcommand{\bra}{\langle}
\newcommand{\ket}{\rangle}

\makeatletter
\def\namedlabel#1#2{\begingroup  
    (#2)%
    \def\@currentlabel{#2}%
    \phantomsection\label{#1}\endgroup
}
\makeatother

\begin{document}

\title{Anomalous recurrence properties of many-dimensional zero-drift random walks}
\author{Nicholas Georgiou\footnote{Department of Mathematical Sciences, Durham University, South Road, Durham DH1 3LE, UK.}
\footnote{Heilbronn Institute for Mathematical Research, School of Mathematics, University of Bristol, University Walk, Bristol, BS8 1TW.}
\and Mikhail V.\ Menshikov\footnotemark[1]
\and Aleksandar Mijatovi\'c\footnote{Department of Mathematics, Imperial College London,
180 Queen's Gate, London SW7 2AZ, UK.}
\and  Andrew R.\ Wade\footnotemark[1]}
 
\date{29 June 2015}
\maketitle

\begin{abstract}
Famously, a $d$-dimensional, spatially homogeneous random walk whose
increments are non-degenerate, have finite second moments, and have zero mean
is recurrent if $d \in \{1,2\}$ but transient if $d \geq 3$. Once spatial homogeneity
is relaxed, this is no longer true. We study a family of zero-drift spatially non-homogeneous random
walks (Markov processes)
whose increment covariance matrix is asymptotically constant along rays from the origin,
and which, in any ambient dimension $d \geq 2$, can be adjusted so that the walk is either transient or recurrent. Natural examples
are provided by random walks whose increments are supported on ellipsoids that are symmetric about the ray from the origin
through the walk's current position; these \emph{elliptic random walks} generalize the classical homogeneous Pearson--Rayleigh walk
(the spherical case). Our proof of the recurrence classification is based on fundamental work of Lamperti.
  \end{abstract}

\medskip


\medskip


\section{Introduction}
\label{sec:intro}

  A $d$-dimensional
  random walk
that
  proceeds via a sequence of unit-length steps, each in an independent and
uniformly random direction,
  is sometimes
  called a {\em Pearson--Rayleigh} random walk (PRRW),
  after the exchange in the letters
  pages of {\em Nature} between
  Karl Pearson and Lord Rayleigh in 1905 \cite{pr}.
Pearson was interested in two dimensions and questions of migration of species (such as mosquitoes) 
\cite{pearson}, although Carazza has speculated that Pearson was a golfer \cite[p.~419]{carazza};
Rayleigh had earlier considered the acoustic `random walks' in phase space produced by combinations of sound waves of the same amplitude and random phases. 

The PRRW can be represented via
 partial sums of sequences of i.i.d.~random vectors that
  are 
   uniformly distributed on the unit sphere
$\Sp{d-1}$ in $\R^d$. Clearly the increments have mean zero, i.e., the PRRW has \emph{zero drift}.
 The PRRW has received some renewed interest recently
as a model for microbe locomotion \cite{berg,nossal,nw}.
Chapter 2 of \cite{hughes} gives a general 
  discussion of these walks,
which  have been well-understood for many years. In particular,   it is well known
that the PRRW is recurrent for $d \in \{1, 2\}$ and transient if $d \geq 3$.
  
Suppose that we replace the spherically symmetric increments of the PRRW by
increments that instead have some \emph{elliptical} structure, while retaining the zero drift. 
For example,
one could take the increments to be uniformly distributed on the surface
of an ellipsoid of fixed shape and orientation, as represented by the picture
on the right of Figure~\ref{fig1}. More generally, one should view the ellipses in
Figure~\ref{fig1} as representing the \emph{covariance} structure of the increments of the walk
(we will give a concrete example later; the uniform distribution on the ellipse
is actually not the most convenient for calculations).

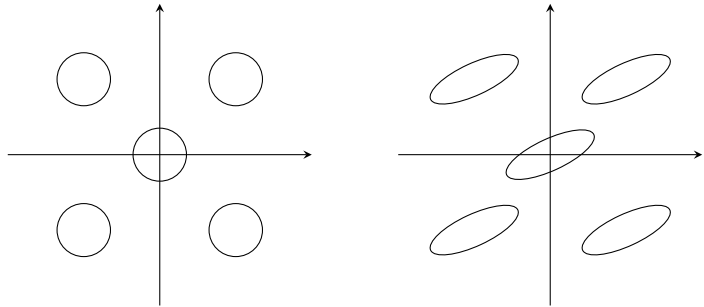
\begin{figure}[!h]
\center
\begin{tikzpicture}
\path[use as bounding box] (-2.5,-2.5) rectangle (2.5,2.5);
\path[->,>=stealth]
  (0,-2) edge (0,2)
  (-2,0) edge (2,0);
\foreach \x/\y in {-1/-1, -1/1, 1/-1, 1/1, 0/0}
  \draw (\x,\y) ellipse (10pt and 10pt);
\end{tikzpicture}
\begin{tikzpicture}
\path[use as bounding box] (-2.5,-2.5) rectangle (2.5,2.5);
\path[->,>=stealth]
 (0,-2) edge (0,2)
 (-2,0) edge (2,0);
\foreach \x/\y in {-1/-1, -1/1, 1/-1, 1/1, 0/0}
  \draw[rotate around={25:(\x,\y)}] (\x,\y) ellipse (18pt and 6pt);
\end{tikzpicture}
\caption{Pictorial representation of spatially homogeneous random walks with increments
distributed on a fixed circle ({\em left}) and a fixed ellipse ({\em right}).}
\label{fig1}
\end{figure}

A little thought shows that the walk represented by the picture on the right
of Figure~\ref{fig1} is  essentially no different to the PRRW: an affine transformation
of $\R^d$ will map the walk back to a walk whose increments have the same covariance structure
as the PRRW. To obtain genuinely different behaviour, it is necessary to abandon spatial homogeneity.

In this paper
  we consider a family of spatially \emph{non-homogeneous} random walks
with zero drift. These include 
generalizations of the PRRW 
in which the increments are not i.i.d.\ but have
a distribution supported on an ellipsoid of fixed size and
  shape but whose
  orientation depends upon the current position of the walk. Figure~\ref{fig2} gives representations
of two important types of example, in which the ellipsoid is aligned so that its
  principal axes are  parallel or perpendicular to
  the vector of the current position of the walk, which sits at the centre of the ellipse.
 
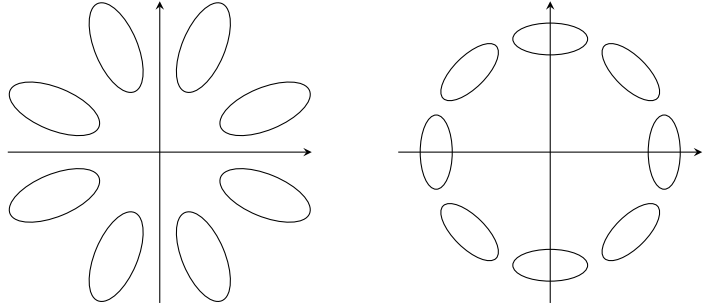
\begin{figure}[!h]
\center
\begin{tikzpicture}
\path[use as bounding box] (-2.5,-2.5) rectangle (2.5,2.5);
\path[->,>=stealth]
  (0,-2) edge (0,2)
  (-2,0) edge (2,0);
\foreach \angle in {22.5, 67.5, ..., 360}
  \draw[rotate=\angle] (0:1.5) ellipse (18pt and 8pt);
\end{tikzpicture}
\begin{tikzpicture}
\path[use as bounding box] (-2.5,-2.5) rectangle (2.5,2.5);
\path[->,>=stealth] (0,-2) edge (0,2)
(-2,0) edge (2,0);
\foreach \angle in {0, 45, ..., 359}
\draw[rotate=\angle] (0:1.5) ellipse (6pt and 14pt);
\end{tikzpicture}
\caption{Pictorial representation of spatially non-homogeneous random walks with increments
distributed on a radially-aligned ellipse with major axis aligned
in the radial sense ({\em left}) and in the transverse sense ({\em right}).}
\label{fig2}
\end{figure}

The random walks represented by Figure~\ref{fig2} are no longer sums
  of i.i.d.~variables. 
These modified walks can behave very differently to the PRRW.
For instance, one of the two-dimensional random walks represented in Figure~\ref{fig2}
is \emph{transient} while the other (as in the classical case) is recurrent. 
The reader who has not seen this kind of example before may take a moment to identify which is which.
It is this \emph{anomalous recurrence behaviour} that is the main subject of the present paper.
In the next section, we give a formal description of our model and state our
 main results.
  
We end this introduction with a brief comment on motivation. 
  In biology, the PRRW is more natural than a lattice-based
 walk for modelling the motion of  microscopic organisms, such as certain bacteria,
 on a surface. Experiment suggests that the locomotion of several kinds of
 cells consists of roughly straight line segments linked by discrete changes
 in direction:
 see, e.g.,~\cite{nossal,nw}. The generalization
 to elliptically-distributed increments  studied here represents movement on a surface
 on which either radial or transverse motion is inhibited.
 In chemistry and
 physics, the trajectory of a finite-step
 PRRW (also called a `random chain')
 is an idealized model of the growth of weakly interacting polymer molecules:
 see, e.g.,~\S2.6 of \cite{hughes}. The modification
 to ellipsoid-supported jumps represents polymer growth in a biased medium.

\section{Model and main results}
\label{sec:rws}

We work in $\R^d$, $d \geq 1$. Our main interest is in $d \geq 2$, as we shall explain shortly.
Write $\be_1, \ldots, \be_d$ for the standard orthonormal basis vectors
in $\R^d$.
 Write $\0$ for the origin in $\R^d$,
and let $\| \, \cdot \, \|$ denote the Euclidean norm and $\bra \,\cdot\, ,\!\, \cdot\, \ket$ the Euclidean inner product on $\R^d$.
Write $\Sp{d-1} := \{ \bu \in \R^d : \| \bu \| = 1\}$ for the unit sphere in $\R^d$.
For $\bx \in \R^d \setminus \{ \0 \}$, set $\hat \bx := \bx / \| \bx \|$;
also set $\hat \0 := \be_1$, for convenience. For definiteness, vectors $\bx \in \R^d$
are viewed as  column vectors throughout.

We now define $X=(X_n , n \in \ZP)$, 
a discrete-time, time-homogeneous Markov process on a (non-empty, unbounded) subset $\X$ of $\R^d$.
Formally, $(\X,\cB_\X)$ is a measurable space, $\X$ is a Borel
subset of $\R^d$, and $\cB_\X$ is the $\sigma$-algebra
of all $B \cap \X$ for $B$ a Borel set  in $\R^d$.
Suppose $X_0$ is some fixed (i.e., non-random) point in $\X$. 
Write
\[ \Delta_n := X_{n+1} - X_n ~~ (n \in \ZP) \]
for the increments of $X$. By assumption, given $X_0, \ldots, X_n$,
the law of $\Delta_n$ depends only on $X_n$ (and not on $n$);
so often we ease notation by taking $n=0$ and writing just $\Delta$ for $\Delta_0$.
We also use the shorthand $\Pr_\bx [ \, \cdot \, ] = \Pr [ \, \cdot \, \! \mid X_0 = \bx]$
for probabilities when the walk is started from $\bx \in \X$; similarly we use $\Exp_\bx$ for the
corresponding expectations.

We make the following moments assumption:
\begin{description}
\item[\namedlabel{ass:moments}{A0}] There exists $p >2$ such that $\sup_{\bx \in \X} \Exp_\bx [ \| \Delta \|^p  ] < \infty$.
\end{description}
The assumption \eqref{ass:moments} ensures that $\Delta$ has a well-defined mean vector $\mu(\bx) := \Exp_\bx [ \Delta ]$, and we suppose that the random walk has \emph{zero drift}:
\begin{description}
\item[\namedlabel{ass:zero_drift}{A1}] Suppose that $\mu(\bx) = \0$ for all $\bx \in \X$.
\end{description}
The assumption \eqref{ass:moments} also ensures that $\Delta$ has a well-defined covariance matrix, which we denote by
$
 M (\bx) := \Exp_\bx [ \Delta \Delta^{\!\tra} ],
$
where $\Delta$ is viewed as a column vector. 
To rule out pathological cases, we assume that $\Delta$ is \emph{uniformly non-degenerate}, in the following sense.
\begin{description}
\item[\namedlabel{ass:unif_ellip}{A2}] There exists $v > 0$ such that $\trace(M(\bx)) = \Exp_\bx[ \| \Delta \|^2 ] \geq v$ for all $\bx \in \X$.
\end{description}
Note that  assumption~\eqref{ass:unif_ellip} is weaker than \emph{uniform ellipticity}, which in this context
usually means,  for some $\eps >0$, $\Pr_\bx[  \Delta \cdot \bu \geq \eps  ] \geq \eps$ for all $\bu \in \Sp{d-1}$ and all $\bx$.

Our main interest is in a recurrence classification. First, we state the following basic `non-confinement'
result.
\begin{proposition}\label{lem:lim_sup_infty}
Suppose that $X$ satisfies assumptions \eqref{ass:moments}, \eqref{ass:zero_drift} and \eqref{ass:unif_ellip}.  Then
\begin{equation}
\label{eqn:limsup=+infty}
\limsup_{n\to\infty} \| X_n \| = +\infty, \as
\end{equation}
\end{proposition}
We give the proof of Proposition~\ref{lem:lim_sup_infty} in Section~\ref{sec:non-confinement}; we actually prove more, namely that the hypotheses of Proposition~\ref{lem:lim_sup_infty}
ensure that a version of Kolmogorov's `other' inequality holds.
The fact \eqref{eqn:limsup=+infty} ensures that questions of the escape of trajectories  to infinity  are non-trivial. 
Indeed, we will give conditions under which one or other of the following two behaviours 
(which are not \emph{a priori}  exhaustive)
occurs:
\begin{itemize}
\item $\lim_{n \to \infty} \| X_n \| = +\infty$, a.s., in which case we say that $X$ is \emph{transient};
\item $\liminf_{n \to \infty} \| X_n \| \leq r_0$, a.s., for some constant $r_0 \in \RP$, when we say $X$ is \emph{recurrent}.
\end{itemize}
If $X$ is an irreducible time-homogeneous Markov chain on a locally finite state-space, these definitions reduce to the usual notions of transience and recurrence;
in general state-spaces, our approach allows us to avoid unnecessary technicalities concerning irreducibility. 

In dimension $d=1$, it is a consequence of the classical Chung--Fuchs theorem (see
\cite{cf} or Chapter 9 of \cite{kall})       
that a spatially \emph{homogeneous} random walk with zero drift is necessarily recurrent.
However, this is \emph{not} true for a spatially non-homogeneous random walk:
as observed by Rogozin and Foss \cite{rf78}, a counterexample is provided by a version of the `oscillating random walk'
of Kemperman \cite{kemperman}
in which the increment law is one of two distributions (with mean zero but infinite second moment) depending on the walk's present sign. Our conditions exclude these heavy-tailed phenomena, so that in $d=1$
recurrence is assured in our setting.

\begin{theorem}
\label{t:zero_drift_implies_recurrence}
 Suppose that $d=1$. Suppose that $X$ satisfies assumptions \eqref{ass:moments}, \eqref{ass:zero_drift}, and \eqref{ass:unif_ellip}.
Then $X$ is recurrent.
\end{theorem}

Theorem~\ref{t:zero_drift_implies_recurrence} is essentially contained in a result of Lamperti~\cite[Theorem 3.2]{lamp1}; we give a self-contained proof below.
Theorem~\ref{t:zero_drift_implies_recurrence} shows that in $d=1$, under mild conditions,
the classical Chung--Fuchs recurrence classification for homogeneous
zero-drift random walks extends to zero-drift non-homogeneous random walks.
The purpose of the present paper is to demonstrate a natural family of examples in dimension $d \geq 2$
where this extension fails, and hence exhibit the following.

\begin{fact}
There exist spatially non-homogeneous random walks 
whose increments
are non-degenerate, have uniformly bounded second moments, and have zero mean,
which are \vskip-3mm
\begin{itemize}
\setlength\itemsep{-0.3em}
\vskip-4mm
\item transient in $d=2$;
\item recurrent in $d \geq 3$.
\end{itemize}
\end{fact}

Although certainly appreciated by experts, this fact is perhaps not as widely known as it might be. 
Zeitouni (pp.~91--92 of \cite{zeit}) describes an example of a transient zero-drift random walk on $\Z^2$, and states that the idea ``goes back to Krylov (in the context of diffusions)''. Peres, Popov and Sousi \cite{pps} investigate the minimal number of different increment distributions required for anomalous recurrence behaviour.

We now introduce our family of non-homogeneous random walks.
Write $\| \, \cdot \, \|_{\rm op}$ for the matrix (operator) norm given by $\| M \|_{\rm op} = \sup_{\bu \in \Sp{d-1}} \| M \bu \|$.
The following assumption on the asymptotic stability of the covariance structure of the process  along rays is central.
\begin{description}
\item[\namedlabel{ass:cov_limit}{A3}] Suppose that there exists a positive-definite matrix function $\sigma^2$
with domain $\Sp{d-1}$ such that, as $r \to \infty$,
\[
\eps(r) := \sup_{\bx \in \X : \| \bx \| \geq r} \| M( \bx ) - \sigma^2 ( \hat\bx ) \|_{\rm op} \to 0 .
\]
\end{description}
A little informally, \eqref{ass:cov_limit} says that $M (\bx) \to \sigma^2 (\hat \bx)$ as $\| \bx \| \to \infty$;
in what follows, 
we will often make similar statements, formal versions of which may be cast as in \eqref{ass:cov_limit}.

Note that \eqref{ass:unif_ellip} and \eqref{ass:cov_limit} together imply that $\trace(\sigma^2(\bu) ) \geq v >0$;  
next we impose a key assumption on the form of  $\sigma^2$ that is considerably stronger. To describe this, it is convenient to introduce the notation $\bra \, \cdot \, , \! \, \cdot \, \ket_{\bu}$ that defines,
for each $\bu \in \Sp{d-1}$, an inner product on $\R^d$ via
\[ \bra \by  , \bz \ket_{\bu} := \by^\tra \cdot \sigma^2 ( \bu ) \cdot \bz = \bra \by, \sigma^2(\bu)\cdot\bz \ket, ~~   \text{for }\by, \bz \in \R^d  .\]
\begin{description}
\item[\namedlabel{ass:cov_form}{A4}] 
Suppose that there exist constants $U$ and $V$ with $0 < U < V < \infty$
 such that, for all $\bu \in \Sp{d-1}$,
\[
\bra \bu , \bu \ket_{\bu} = U,  ~~\text{and}~~
\trace(\sigma^2(\bu) ) = V.
\] 
\end{description}
Informally, $V$ quantifies the total variance of the increments, while $U$ quantifies the
 variance in the radial direction; necessarily $U \leq V$. The assumption that $0 \neq U \neq V$
excludes some degenerate cases. As we will see, one possible way to satisfy condition \eqref{ass:cov_form} is to suppose
that the eigenvectors of $\sigma^2(\bu)$ are all parallel or perpendicular to the vector
$\bu$, and that the corresponding eigenvalues are all constant as $\bu$ varies; the level sets of the corresponding
quadratic forms $q_\bu(\bx) := \bra \bx,\bx \ket_\bu$ for $\bu \in \Sp{d-1}$ are then
ellipsoids like those depicted in Figure~\ref{fig2}.

Our main result is the following, which shows that
 both transience and recurrence are possible for \emph{any} $d \geq 2$,
 depending on parameter choices;
as seen in Theorem~\ref{t:zero_drift_implies_recurrence},
this possibility of anomalous recurrence behaviour
is a genuinely multidimensional phenomenon under our regularity conditions. 

\begin{theorem}
\label{thm:recurrence}
Suppose that $X$ 
satisfies \eqref{ass:moments}--\eqref{ass:cov_form}, with constants $0 < U < V$ as defined in \eqref{ass:cov_form}.  The following recurrence
classification is valid.
\begin{itemize}
\item[(i)] If $2U < V$, then $X$ is transient.
\item[(ii)] If $2U > V$, then $X$ is recurrent.
\item[(iii)]  If $2U = V$ and \eqref{ass:cov_limit} holds with $\eps(r) = O(r^{-\delta_0})$ for some $\delta_0 > 0$, then  $X$ is recurrent.
\end{itemize}
\end{theorem}

Moreover, we show that in any of the above cases, $X$ is \emph{null} in the following sense.
\begin{theorem}
\label{thm:null}
Suppose that $X$ 
satisfies \eqref{ass:moments}--\eqref{ass:cov_form}, with constants $0 < U < V$ as defined in \eqref{ass:cov_form}. Then, in any of the cases (i)--(iii) in Theorem~\ref{thm:recurrence}, for any bounded $A \subset \R^d$,  
\begin{equation}
\label{eq:null}
 \lim_{n \to \infty} \frac{1}{n} \sum_{k=0}^{n-1} \1 { X_k \in A } = 0, \as  \text{ and in } L^q \text{ for any } q \geq 1. 
\end{equation}
\end{theorem}

\begin{remark}
\label{rem:UequalsV}
Theorems~\ref{thm:recurrence} and \ref{thm:null} both remain valid if we permit  $V = U >0$ in \eqref{ass:cov_form};
indeed, the condition $U < V$ is not used in the proof of Theorem~\ref{thm:recurrence} given below, so this case is recurrent, by Theorem~\ref{thm:recurrence}(ii).
The condition $U < V$ is used at one point to simplify the proof of Theorem~\ref{thm:null} given below, but a small
modification of the argument also works in the case $U=V$.
\end{remark}

The remainder of the paper is organised as follows.  In Section~\ref{sec:ell-rws} we
describe a specific family of examples called \emph{elliptic random walk models} that satisfy
assumptions \eqref{ass:moments}--\eqref{ass:cov_form} and exhibit both transient and recurrent behaviour
dependent on the parameters of the model.  We also present some simulated data that depicts the
random walks in both cases.  In Section~\ref{sec:non-confinement} we prove a
$d$-dimensional martingale version of Kolmogorov's other inequality and use that to prove the
non-confinement result (Proposition~\ref{lem:lim_sup_infty}).  In Section~\ref{sec:rw} we
prove the recurrence classification (Theorem~\ref{thm:recurrence}), and in
Section~\ref{sec:nullity} we prove Theorem~\ref{thm:null}.  In the appendix we prove
recurrence in the one-dimensional case (Theorem~\ref{t:zero_drift_implies_recurrence}).

Finally, we remark that in work in progress  we investigate diffusive scaling limits for random walks of the type described in the present paper;
the diffusions that appear as scaling limits possess certain pathologies from the point of view
of diffusion theory that make them interesting in their own right.

 \section{Example: Elliptic random walk model} 
  \label{sec:ell-rws}

Let $d \geq 2$.
We describe a specific model on $\X = \R^d$ where the jump distribution at $\bx \in \R^d$ is supported on an ellipsoid having one distinguished axis aligned with the vector $\bx$.
The model is specified by two constants $a,b >0$.  Construct $\Delta$ as follows.
Given $X_0 = \bx$, take $\bzeta$ uniform on $\Sp{d-1}$ and set
\begin{equation}\label{eqn:Delta-d-dim}
\Delta = Q_{\hat{\bx}} D \bzeta
\end{equation}
for $Q_{\hat{\bx}}$ an orthogonal matrix representing a transformation of $\R^d$ mapping
$\be_1$ to $\hat\bx$, and $D = \sqrt{d} \diag { a, b , \ldots, b } $.   See Figure~\ref{fig:Delta}. 

\begin{figure}[!h]
\begin{center}
\begin{tikzpicture}[thick]
\path[use as bounding box] (-2,-2.2) rectangle (10,2);

\draw (-0.5,0) circle (1);
\draw (-1.5,0) arc (180:360:1 and 0.4);
\draw[dotted] (0.5,0) arc (0:180:1 and 0.4);
\node at (-0.3,0.65) {$\bzeta$};
 
\draw (4,0) ellipse (1.2 and 2);
\draw (2.8,0) arc (180:360:1.2 and 0.6);
\draw[dotted] (5.2,0) arc (0:180:1.2 and 0.6);
\node at (4.2,1.2) {$D\bzeta$};

\draw[rotate around={-30:(9,0)}] (9,0) ellipse (1.2 and 1.8);
\draw[rotate around={-30:(9,0)}] (7.8,0) arc (180:360:1.2 and 0.75);
\draw[rotate around={-30:(9,0)},dotted] (10.2,0) arc (0:180:1.2 and 0.75);

\path[->,>=stealth] 
(-0.5,0) edge (0.1,0.6)
(4,0) edge (4.72,1.2)
[rotate around={-30:(9,0)}] (9,0) edge (9.72,1.08);

\node at (10,0.85) {$\Delta$};
\node at (9.2,-0.3) {$\bx$};
\node at (7.2,-2.5) {$\0$};
\path[thin, dashed, rotate around={-30:(9,0)}] (9,0) edge (9,-3);
\draw[fill=black, rotate around={-30:(9,0)}] (9,-3) circle (1pt);

\draw[double equal sign distance,-implies] (1,0) -- (2.4,0);
\draw[double equal sign distance,-implies] (5.7,0) -- (7.1,0);
\end{tikzpicture}
\end{center}
\caption{Definition of $\Delta = Q_{\hat \bx}D\bzeta$.}\label{fig:Delta}
\end{figure}
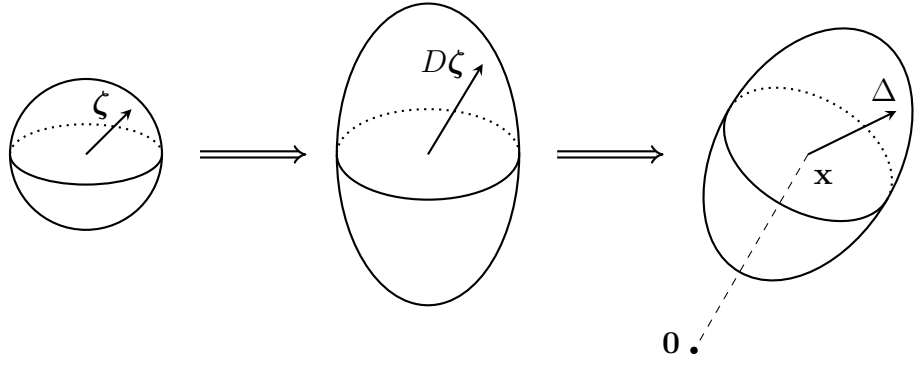

(Recall that $\hat\0 = \be_1$, so for $X_0 = \0$ we can take $Q_{\hat\bx} = I$ and $\Delta
= D\bzeta$.)  Thus $\Delta$ is a random point on an ellipsoid that has one distinguished
semi-axis, of length $a\sqrt{d}$, aligned in the $\hat\bx$ direction, and all other
semi-axes of length $b\sqrt{d}$.  Note that the law of $\Delta$ is well defined owing to
the spherical symmetry of the uniform distribution on $\Sp{d-1}$ and the fact that only
one axis of the ellipsoid is distinguished (for this reason it is enough to take any
$Q_{\hat{\bx}}$ satisfying $Q_{\hat{\bx}} \be_1 = \hat\bx$ in order to define $\Delta$;
see also Remark~\ref{rem:Hx} below).

Note also that $\Delta$ is not chosen to be uniformly distributed on the surface of the ellipsoid; this does not affect the range of asymptotic behaviour exhibited by the family of walks as $a$ and $b$ vary, but it does simplify the calculation of $M(\bx)$.  Indeed,
we have
\[
M(\bx) = \Exp_\bx[ \Delta \Delta^\tra] = \Exp[ Q_{\hat{\bx}} D \bzeta \bzeta^\tra D Q_{\hat{\bx}}^\tra ] = Q_{\hat{\bx}} D \Exp[ \bzeta \bzeta^\tra ] D Q_{\hat{\bx}}^\tra = \frac{1}{d} Q_{\hat{\bx}} D^2 Q_{\hat{\bx}}^\tra,
\]
by linearity of expectation, and using the fact that $\Exp[\bzeta\bzeta^\tra] = \frac{1}{d}I$ for $\bzeta$ uniformly distributed on $\Sp{d-1}$.
Also, a calculation similar to the above confirms that $\mu(\bx) = \0$ for all $\bx \in \R^d$, since $\Exp[\bzeta] = \0$.

Since $\| \Delta \|$ is bounded above by $\sqrt{d}\max\{a,b\}$, assumption
\eqref{ass:moments} holds.  Clearly \eqref{ass:zero_drift} and \eqref{ass:cov_limit} hold,
with $\sigma^2(\bu) =\frac{1}{d} Q_\bu D^2 Q_\bu^\tra$ for $\bu \in \Sp{d-1}$.  It is also
a simple matter to check that \eqref{ass:unif_ellip} and \eqref{ass:cov_form} hold: the matrix $\sigma^2(\bu)$ represented in coordinates for the orthonormal basis $\{ Q_\bu \be_1 = \bu, Q_\bu \be_2, \dots , Q_\bu \be_d \}$ is diagonal with entries $a^2, b^2, \dots, b^2$.  Indeed,
\[
\begin{split}
\sigma^2(\bu) = \frac{1}{d} Q_\bu D^2 Q_\bu^\tra &= Q_\bu [ b^2 I + (a^2-b^2) \be_1\be_1^\tra ] Q_\bu^\tra\\
&= a^2 \bu\bu^\tra  + b^2( I - \bu\bu^\tra ),
\end{split}
\]
and therefore $\bra \bu, \bu \ket_\bu = \bra \bu , \sigma^2(\bu) \cdot \bu\, \ket = a^2 > 0$
for all $\bu \in \Sp{d-1}$, and $\trace{(M(\bx))} = \trace{(\sigma^2(\hat\bx))} = a^2 +
(d-1)b^2 > 0$ for all $\bx \in \R^d$.

\begin{remark}\label{rem:Hx}
The seeming ambiguity in the definition of $\Delta$ due to the choice of $Q_{\hat\bx}$ can
be resolved by noting that $\Delta$ can be rewritten as
\[
\Delta = Q_{\hat\bx} D Q_{\hat\bx}^\tra Q_{\hat\bx} \bzeta = Q_{\hat\bx} D
Q_{\hat\bx}^\tra \tilde\bzeta,
\]
where $\tilde\bzeta = Q_{\hat\bx} \bzeta$ is also uniform on $\Sp{d-1}$ (this follows from the
spherical symmetry of the uniform distribution on $\Sp{d-1}$).  Moreover, the
symmetric matrix $H_{\hat\bx} := Q_{\hat\bx} D Q_{\hat\bx}^\tra$ is determined explicitly
in terms of $\hat\bx$:
\[
\begin{split}
H_{\hat\bx} = Q_{\hat\bx} D Q_{\hat\bx}^\tra &= Q_{\hat\bx}( b\sqrt{d} I + (a-b)\sqrt{d}\be_1\be_1^\tra) Q_{\hat\bx}^\tra \\
&= b\sqrt{d} I + (a-b)  \sqrt{d} \hat\bx \hat\bx^\tra.
\end{split}
\]
Consequently, we could choose to specify $\Delta$ explicitly as
\[
\Delta = H_{\hat\bx}\tilde\bzeta = b\sqrt{d}\tilde\bzeta + (a-b)\sqrt{d}\hat\bx\bra\hat\bx,\tilde\bzeta\ket,
\]
 with $\tilde\bzeta$ taken to be uniform on $\Sp{d-1}$.
As before, we find that $\Exp_{\bx}[\Delta] = H_{\hat\bx} \Exp[ \tilde\bzeta ] = \0$
and
\[
\Exp_\bx[ \Delta\Delta^\tra ] = H_{\hat\bx} \Exp[ \tilde\bzeta\tilde\bzeta^\tra
]H_{\hat\bx} = \textstyle\frac{1}{d}H_{\hat\bx}^2 = a^2\hat\bx\hat\bx^\tra + b^2( I -
\hat\bx\hat\bx^\tra ).
\]
\end{remark}

Recall that we assume our random walk to be time-homogeneous, so that
equation~\eqref{eqn:Delta-d-dim} in fact determines the distribution of $\Delta_n$ for all
$n \geq 0$.  Formally, we define $\bzeta_0, \bzeta_1,\dots$ a
sequence of independent random variables uniformly distributed on $\Sp{d-1}$, and for each
$n \geq 0$ we define $\Delta_n$ conditional on $\{X_n = \bx \}$ via
\begin{equation}\label{eqn:Delta-n-d-dim}
\Delta_n = Q_{\hat{\bx}} D \bzeta_n.
\end{equation}
We call $X = (X_n, n \in \ZP)$ defined in this way an \emph{elliptic random walk}.

As a corollary to Theorems~\ref{thm:recurrence} and \ref{thm:null}, we get the following recurrence classification for the elliptic random walk model. For this model the $\eps(r)$ in \eqref{ass:cov_limit} is identically zero so we get a complete classification that includes the boundary case.

   \begin{corollary}
   \label{cor:ellipsoid-Lamperti}
Let $d \geq 2$ and $a, b \in (0,\infty)$. Let $X$ be an elliptic random walk on $\R^d$. Then
   $X$ is transient if $a^2 < (d-1)b^2$ and null-recurrent if $a^2 \geq (d-1)b^2$.
   \end{corollary}

In two dimensions we can explicitly describe the random walk as follows.  For $\bx \in \R^2$, $\bx \neq \0$ with $\bx = (x_1, x_2)$ in Cartesian components, set $\bx^\per :=
 (-x_2, x_1)$.
  Fix $a,b \in (0,\infty)$.
  Let $E_\bx (a,b)$ denote the ellipse 
  with centre $\bx$ and
   principal axes aligned
  in the $\bx$, $\bx^\per$ directions, with lengths
  $2\sqrt{2}a$, $2\sqrt{2}b$ respectively,
  given in parametrized form by  
  \begin{equation}
\label{param}
 E_\bx (a,b) := \left\{ \bx + \sqrt{2}a \frac{\bx}{\| \bx \|} \cos \phi
  + \sqrt{2}b \frac{\bx^\per}{\| \bx\|} \sin \phi : \phi \in (-\pi,\pi] \right\},
\end{equation}
and for $\bx = \0$ set
\[
E_\0 (a,b) := \left\{ \sqrt{2}a\, \be_1 \cos \phi + \sqrt{2}b\, \be_2 \sin \phi : \phi \in (-\pi,\pi] \right\}.
\]
The parameter $\phi$ in the parametrization (\ref{param}) should be interpreted with caution: it
is \emph{not}, in general, the central angle of the parametrized point on the ellipse.  

Given
  $X_n = \bx \in \R^2$,  $X_{n+1}$ is 
taken to be 
  distributed on $E_{\bx}(a,b)$,
`uniformly' with respect to the parametrization (\ref{param}).
Precisely, let $\phi_0, \phi_1, \ldots$ be
  a sequence of independent random variables
  uniformly distributed on $(-\pi,\pi]$. Then, on $\{ X_n \neq \0 \}$,
  \begin{equation}
  \label{elljumps}
   X_{n+1} = X_n + \sqrt{2}a \frac{X_n}{\| X_n\|} \cos \phi_n 
  + \sqrt{2}b \frac{X_n^\per}{\| X_n \|} \sin \phi_n ,\end{equation}
  while, on $\{ X_n = \0 \}$, 
  \begin{equation}
  \label{jump0}
  X_{n+1} = (\sqrt{2}a \cos \phi_n, \sqrt{2}b \sin \phi_n ) .\end{equation}

Figure~\ref{fig:sim} shows two sample paths of a simulation of the elliptic random walk in
$\R^2$ in the two cases of recurrence and transience.  In each picture the walk starts at the
origin at the centre of the picture; time is represented by the variation in colour (from
red to yellow, or from dark to light if viewed in grey-scale).
\begin{figure}[!h]
\begin{center}
\includegraphics[scale=1]{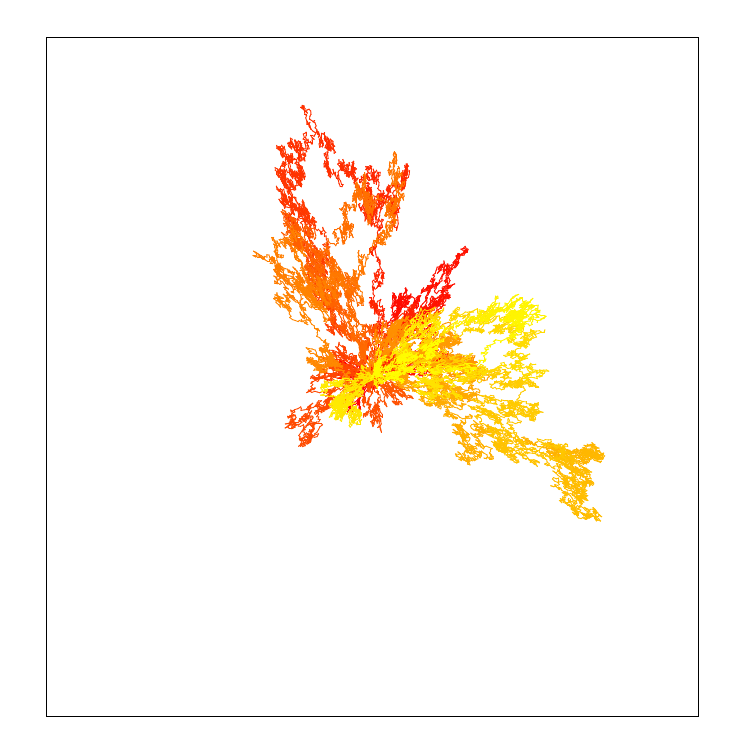}
\includegraphics[scale=1]{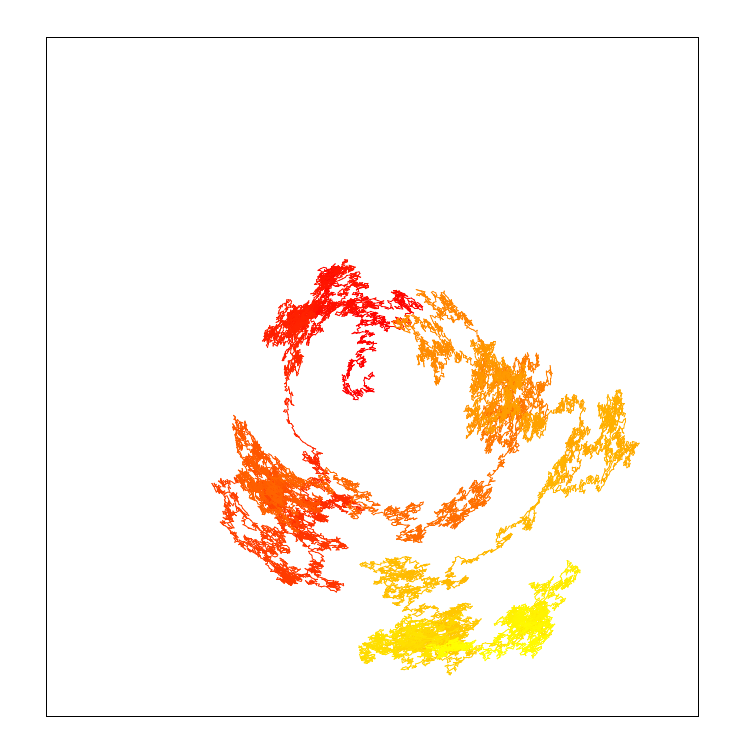}
\end{center}
\caption{Simulation of the elliptic random walk in $\R^2$ for the recurrent case $a > b$
  (\emph{left}) and the transient case $a < b$ (\emph{right}).}\label{fig:sim}
\end{figure}

\begin{remarks}\label{rem:ellipse}
\begin{enumerate}[(a)] 
\item   The process $X$ reduces to the classical PRRW when 
  $a = b$: in that case it is spatially homogeneous, i.e.,
  the distribution of the increment $X_{n+1}-X_n$ does not depend on $X_n$. For
  $a \neq b$ the random walk is not spatially homogeneous, and the jump
  distribution depends upon the projection onto the unit sphere
  of the walk's current position.

\item As mentioned earlier, we choose to take increments as defined at \eqref{elljumps},
  rather than increments that are uniform on the ellipse with respect to one-dimensional
  Lebesgue measure on $E_\bx (a,b)$, purely for computational reasons.  In fact, in two
  dimensions, since the Lebesgue measure on $E_\bx (a,b)$ coincides with the measure
  induced by taking $\phi$ uniformly distributed on $(-\pi,\pi]$ when $a = b$, and the
  case $a=b$ is critically recurrent, the qualitative behaviour will be the same in either
  case: the walk will be transient for $a<b$ and recurrent for $a \geq b$.  For higher
  dimensions, taking increments that are uniform with respect to the Lebesgue measure on
  $E_\bx^d(a,b) := \{ Q_{\hat\bx}D\bu : \bu \in \Sp{d-1} \}$ will still specify a family
  of models that exhibit a phase transition, from transience (for $a/b$ small) to
  recurrence (for $a/b$ large) but the exact shape of the ellipsoid in the critical case
  (i.e., the smallest ratio $a/b$ for which the walk is recurrent) may be different.

\item It follows from \eqref{eqn:Delta-n-d-dim} that
\begin{align}
\|X_{n+1}\|^2 &= \|X_n\|^2 + 2\|X_n\|\bra \widehat{X}_n , \Delta_n \ket + \| \Delta_n \|^2
  \nonumber\\
&= \|X_n\|^2 + 2\|X_n\|\bra \be_1, D \bzeta_n \ket + \bra \bzeta_n, D^2 \bzeta_n \ket \nonumber\\
&= \|X_n\|^2 + 2a\sqrt{d}\|X_n\|\bra \be_1, \bzeta_n \ket + (a^2-b^2)d \bra \be_1 , \bzeta_n \ket^2 + b^2d.
\label{eqn:norm-X}
\end{align}
In particular, for this family of models $(\|X_n\|,n \in \ZP)$ is itself a Markov process,
since the distribution of $\|X_{n+1}\|$ depends only on $\|X_n\|$ and not $X_n$; however,
in the general setting of Section~\ref{sec:rws}, this need not be the case. 

One-dimensional processes with evolutions reminiscent to that given by \eqref{eqn:norm-X}
have been studied previously by Kingman~\cite{kingman} and Bingham~\cite{bingham}.  Those processes
can be viewed, respectively, as the distance from its start point of a random walk in
Euclidean space, and the geodesic distance from its start point of a random walk on the
surface of a sphere, but in both cases the increments of the random walk have the property
that the distribution of the jump vector is a product of the independent marginal
distributions of the length and direction of the jump vector.  In contrast, for the
elliptic random walk the laws of $\|\Delta_n\|$ and $\bra \widehat X_n , \widehat\Delta_n \ket$
are \emph{not} independent (except when $a=b$).

\item\label{rem:Stas} The theory equally applies to the case where the ellipsoid specifying the jump distribution is oriented with some fixed angle $\alpha \in [0,\pi)$ with respect to the radial direction.  If we define $\Delta = Q^\alpha_{\hat \bx} D \bzeta$, where $Q^\alpha_{\hat \bx}$ is an orthogonal matrix that maps $\be_\alpha := \be_1 \cos{\alpha} + \be_2 \sin{\alpha}$ to $\hat \bx$, then we find that transience of $X$ is equivalent to
\[
(a^2 - b^2) \cos{2 \alpha} < (d-2) b^2.
\]
Note that for $d=2$, $Q^\alpha_{\hat\bx}$ and therefore $\Delta$ are well defined, but this is not so for higher dimensions.   Nevertheless, for \emph{any} collection of matrices $(Q^\alpha_{\bu}; \bu \in \Sp{d-1})$ satisfying $Q^\alpha_\bu \be_\alpha = \bu$ for all $\bu \in \Sp{d-1}$ we get the same recurrence classification. 
 This is because the distribution of $\| X_{n+1} \|$ given $X_n$ is determined through the angle $\alpha$ via
\[
 \|X_n\|^2 + 2\sqrt{d}\|X_n\|(a \bra \be_1, \bzeta_n \ket \cos{\alpha} + b \bra \be_2, \bzeta_n \ket \sin{\alpha} ) + (a^2-b^2)d\bra \be_1,\bzeta_n\ket^2 + b^2d,
\]
and therefore assumption \eqref{ass:cov_form} holds with $U = a^2\cos^2{\alpha} +
b^2\sin^2{\alpha}$ and $V = a^2 + (d-1)b^2$.
\end{enumerate}
\end{remarks}

\section{Non-confinement}
\label{sec:non-confinement}

In this section we prove that 
the assumptions \eqref{ass:moments}, \eqref{ass:zero_drift}, and \eqref{ass:unif_ellip} imply that
$\limsup_{n \to \infty} \|X_n\| = +\infty$, a.s.
We first present a general result for martingales on $\R^d$ satisfying a ``uniform dispersion'' condition; the result can be viewed as a $d$-dimensional martingale
version of \emph{Kolmogorov's other inequality} (see e.g.~\cite[pp.~123, 502]{gut}).

\begin{lemma}\label{lem:d-dim-KOI}
Let $d \in \N$. Suppose that $(Y_n, n \in \ZP)$
is an $\R^d$-valued process adapted to a filtration $(\cG_n , n \in \ZP)$,
with $\Pr [ Y_0 = \0 \mid \cG_0 ] =1$. 
Suppose that there exist $p>2, v>0, B <\infty$ such that for all $n \in \ZP$, a.s.,
\begin{align}
\Exp[ \| Y_{n+1} - Y_n \|^p \mid \cG_n ] &\leq B; \label{KOI:moments}\\
\Exp[ \| Y_{n+1} - Y_n \|^2 \mid \cG_n ] &\geq v;  \label{KOI:unif-ellip} \\
\Exp[   Y_{n+1} - Y_n  \mid \cG_n ] & = \0.       \label{KOI:zero-drift} 
\end{align}
Then there exists $D <\infty$, depending only on $B$, $p$, and $v$, such that for all $n\in\ZP$ and all $x\in \RP$,
\[
\Pr\Bigl[ \max_{0\leq \ell \leq n} \| Y_\ell \| \geq x \Bigmid \cG_0 \Bigr] \geq 1 - \frac{D(1+x)^2}{n}, \as
\]
\end{lemma}  
\begin{proof}
Let $x>0$ and set $\tau = \min\{ n \geq 0 : \| Y_n \| \geq x \}$;
throughout the paper we adopt the usual convention $\min \emptyset := \infty$. 
 In analogy with previous notation, write $\Delta_n = Y_{n+1} - Y_n$ for the jump distribution, and let
\[
W_n = \begin{cases}
Y_n &\text{if $\|Y_n\| \leq A(1+x)$},\\
Y_{n-1} + \widehat{\Delta}_{n-1}(A-1)(1+x) &\text{if $\|Y_n\| > A(1+x)$},
\end{cases}
\]
where $A>1$ is a constant to be specified later. Note that $W_n$ is $\cG_n$-measurable.

Now, on $\{ \|Y_n\| \leq x \}$, $W_n = Y_n$ and
 \begin{align*}
\Exp[W_{n+1} - W_n \mid \cG_n ] = {} &\Exp[ \Delta_n \mid \cG_n ]\\
& {} + \Exp[ \widehat{\Delta}_n\left((A-1)(1+x) - \|\Delta_n\|\right)\1{\|Y_{n+1}\| > A(1+x)} \mid \cG_n ]. 
\end{align*}

But $\{\|Y_{n+1}\| > A(1+x)\} \cap \{\|Y_n\| \leq x \}$ implies that $\|\Delta_n\| > (A-1)(1+x)$, and by
\eqref{KOI:zero-drift}, $\Exp[\Delta_n \mid \cG_n] = \0$.  Hence, on $\{\|Y_n\| \leq x\}$,
 \begin{align*}
\bigl\| \Exp[ W_{n+1} - W_n \mid \cG_n ] \bigr\| &\leq \Exp[ \|\Delta_n\| \1{ \|\Delta_n\| > (A-1)(1+x)} \mid \cG_n]\\
&\leq (A-1)^{-1}(1+x)^{-1} \Exp[ \|\Delta_n\|^2 \mid \cG_n ]\\
&\leq B'(A-1)^{-1}(1+x)^{-1}, \as,
\end{align*}
where, by \eqref{KOI:moments} and Lyapunov's inequality, $B' < \infty$ depends
 only on $B$ and $p$.  Hence we can choose $A \geq A_0$ for some $A_0 = A_0 (B, p, v)$ large enough
 so that
\begin{equation}
\label{eq:w-drift}
\bigl\| \Exp[ W_{n+1} - W_n \mid \cG_n ] \bigr\| \leq (v/8)(1+x)^{-1}, \text{ on } \{ \| Y_n \| \leq x \} .
\end{equation}
\label{eq:w-square}
Also, on  $\{\|Y_n\| \leq x\}$, by a similar argument,
 \begin{align}
\Exp[ \|W_{n+1} - W_n\|^2 \mid \cG_n ] \bigr\| &= \Exp[ \|\Delta_n\|^2 \mid \cG_n ] \nonumber\\
&\quad + \Exp[ \left( (A-1)^2(1+x)^2 - \|\Delta_n\|^2 \right)\1{ \|Y_{n+1}\| > A(1+x) } \mid \cG_n ] \nonumber\\
&\geq \Exp[ \|\Delta_n\|^2 \mid \cG_n ] - \Exp[ \|\Delta_n\|^2 \1{ \| \Delta_{n}\| > (A-1)(1+x) } \mid \cG_n ] \nonumber\\
&\geq v - (A-1)^{2-p}(1+x)^{2-p} \Exp[ \|\Delta_n\|^p \mid \cG_n ]\nonumber\\
&\geq v/2,
\end{align}
for all $x \geq 0$ and $A \geq A_1$ for
sufficiently large $A_1 = A_1 (B, p, v)$, 
using \eqref{KOI:moments} and \eqref{KOI:unif-ellip}.

Now, set $Z_n = \| W_{n \wedge \tau} \|^2$.  Then, on $\{ n < \tau \}$, by \eqref{eq:w-drift} and \eqref{eq:w-square},
\begin{align*}
\Exp[ Z_{n+1} - Z_n \mid \cG_n ] &= \Exp[ \|W_{n+1}\|^2 - \|W_n\|^2 \mid \cG_n ]\\
&= \Exp[ \|W_{n+1} - W_n\|^2 \mid \cG_n ] + 2 \bigl\bra W_n , \Exp[ W_{n+1} - W_n \mid \cG_n ] \bigr\ket\\
&\geq \frac{v}{2} - \frac{ 2 \|W_n\| v }{ 8(1+x) } \geq \frac{v}{2} - \frac{ v x }{ 4(1+x) } \geq \frac{v}{4}.
\end{align*}
Hence $Z_n - \sum_{k=0}^{n-1} v_k$ is a  $\cG_n$-adapted submartingale, where 
\[ v_k = \frac{v}{4} \1{k < \tau} \geq \frac{v}{4}  \1{n < \tau} , ~~ \text{for} ~ 0 \leq k < n  . \] 
 By construction, $0 \leq Z_n \leq A^2(1+x)^2$, so
\[
0 = \Exp [ Z_0 \mid \cG_0 ] \leq \Exp [ Z_n \mid \cG_0 ] - \sum_{k=0}^{n-1} \Exp [ v_k \mid \cG_0] 
 \leq A^2(1+x)^2 -  \sum_{k=0}^{n-1} \frac{v}{4} \Pr [ n < \tau \mid \cG_0 ],
\]
which implies $n(v/4)\Pr[ n < \tau \mid \cG_0 ] \leq A^2(1+x)^2$.  In other words,
\[
\Pr\Bigl[ \max_{0\leq \ell \leq n} \| Y_\ell \| < x \Bigmid \cG_0 \Bigr] \leq \frac{4A^2(1+x)^2}{v n} , \as \qedhere
\]
\end{proof}

Now we can give the proof of Proposition~\ref{lem:lim_sup_infty}.

\begin{proof}[Proof of Proposition~\ref{lem:lim_sup_infty}.]
It is enough to show that for all $x \in \RP$ the event $\{ \| X_n \| \geq x \}$ occurs
infinitely often.  For a given $x$, we will apply Lemma~\ref{lem:d-dim-KOI} to $Y_n = X_{m+n} - X_m$ with $\cG_n = \sigma(X_0,\dots,X_{m+n})$;
that result is applicable, since \eqref{ass:moments}, \eqref{ass:zero_drift} and \eqref{ass:unif_ellip}
imply \eqref{KOI:moments}, \eqref{KOI:zero-drift} and \eqref{KOI:unif-ellip}, respectively. Thus Lemma~\ref{lem:d-dim-KOI} shows that, for some finite $t = t(x)$, 
\begin{equation}
\label{eq:escape-chance}
\Pr{\Bigl[ \max_{0 \leq \ell \leq t-1} \| X_{m+\ell} - X_m \| \geq 2x \Bigmid  X_0,\dots,X_m  \Bigr]} \geq \frac{1}{2}, \as, 
\end{equation}
for all $m \geq 0$. 
For $k=1,2,\dotsc$, define the event
\[
A_k = \Bigl\{ \max_{0 \leq \ell \leq t-1} \|X_{(k-1)t + \ell} - X_{(k-1)t} \| \geq 2x \Bigr\},
\]
and filtration $\cG'_{k-1} = \sigma(X_0,\dots,X_{(k-1)t})$.  Then $A_k \in \cG'_k$, and, by \eqref{eq:escape-chance},
 $\Pr{[ A_k \mid \cG'_{k-1}]} \geq \frac{1}{2}$, a.s., 
for all $k$. An application of L\'evy's extension of the Borel--Cantelli lemma (see, e.g.,~\cite[Cor.~7.20]{kall}) shows that $A_k$ occurs infinitely often, a.s.
For each $k$ such that $A_k$ occurs, either
\begin{itemize}
\item $\| X_{(k-1)t} \| \geq x$, or
\item $\| X_{(k-1)t} \| \leq x$ and $\|X_n \| \geq x$ for some $(k-1)t < n < k t$.
\end{itemize}
Since one of these cases must occur for infinitely many $k$, we have that $\{ \| X_n \| \geq x\}$ occurs infinitely often, as required.
\end{proof}

\section{Recurrence classification}
\label{sec:rw}

In this section we study the random walk $X_n$ and give the proof of the recurrence classification, Theorem~\ref{thm:recurrence}. 
The method of proof is based on applying classical results of Lamperti~\cite{lamp1} to the $\RP$-valued radial process given by $R_n := \| X_n \|$.
The method rests on an analysis of the increments $R_{n+1} - R_n$ given $X_n =  \bx \in \X$;
in general, $R_n$ is not itself a Markov process.
 The following notation will be useful.
Given $\bx \neq \0$ and $\by \in \R^d$, write
\[
\by_\bx := \frac{\bra \bx, \by \ket}{\|\bx\|} = \bra \hat \bx, \by \ket,
\]
so that $\by_\bx$ is the component of $\by$ in the $\hat \bx$ direction, and $\by - \by_\bx\, \hat \bx$ is a vector perpendicular to $\hat \bx$. 

First we state a general result on the increments of $R_n$ for a Markov process $X_n$ on $\X$. Recall that we write $\Delta = X_1 - X_0$, and let $\Delta_\bx$ be the radial component of $\Delta$ at $X_0 = \bx$ in accordance with the notation described above; no confusion should arise with our notation $\Delta_n$ defined previously.

We make an important comment on notation.
When we write $O( \|\bx\|^{-1-\delta} )$, and similar expressions, these are understood to be uniform in $\bx$.  That is, if $f : \R^d \to \R$ and
$g : \RP \to \RP$, we write $f ( \bx ) = O ( g(\|\bx\|) )$ to mean that there exist $C \in \RP$ and $r \in \RP$ such that
\begin{equation}\label{eqn:big-O}
 | f (\bx) | \leq C  g( \| \bx \|)   \text{ for all } \bx \in \X \text{ with } \| \bx \| \geq r .
\end{equation}

  \begin{lemma}
  \label{lemma1}
 Suppose that $X$ is a discrete-time, time-homogeneous Markov process on $\X \subseteq \R^d$ satisfying \eqref{ass:moments} for some $p >2$.
   Then, for $R_n := \| X_n \|$, we have
  \begin{equation}
  \label{bound}
   \sup_{\bx \in \X} \Exp [ | R_{n+1} - R_n |^p \mid X_n = \bx ] < \infty,\end{equation}
and the radial increment moment functions satisfy
   \begin{align}
    \label{mu1}
 \mu_1(\bx) & :=  \Exp [ R_{n+1} - R_n \mid X_n = \bx ] 
  = \Exp_\bx [ \Delta_\bx  ] +
     \frac{ \Exp_\bx [ \| \Delta \|^2 - \Delta^2_{\bx} ] }
    {2 \| \bx \|} + O (\| \bx \|^{-1-\delta} ) ,\\
    \label{mu2}
 \mu_2(\bx) & :=  \Exp [ (R_{n+1}- R_n)^2 \mid X_n = \bx ]
    = \Exp_\bx [ \Delta_\bx^2 ] + O( \|\bx\|^{-\delta} ) , 
    \end{align}  
  as
  $\| \bx \| \to \infty$, for some $\delta=\delta(p) > 0$.
  \end{lemma}
  \begin{proof}
By time-homogeneity, it suffices to consider the case $n=0$. By the triangle inequality,
$| R_{1} - R_0 | = \bigl| \| X_0 + \Delta \| - \| X_0 \| \bigr| \leq \| \Delta \|$, so that \eqref{bound} follows from \eqref{ass:moments}.

We prove \eqref{mu1} and \eqref{mu2} by approximating
\begin{equation}\label{eqn:|x+Delta|-|x|}
\begin{split}
\| \bx + \Delta \| - \| \bx \| &= \sqrt{ \bra \bx+\Delta , \bx + \Delta \ket } - \|\bx \|\\
&= \| \bx \| \left[ \left( 1 + \frac{2 \Delta_\bx}{\|\bx\|} + \frac{\|\Delta\|^2}{\|\bx\|^2} \right)^{1/2} - 1 \right]
\end{split}
\end{equation}
for large $\bx$.  
Let $A_\bx = \{ \| \Delta \| \leq \| \bx \|^\beta \}$ for some $\beta \in (0,1)$ to be determined later.  
On the event $A_\bx$ we approximate \eqref{eqn:|x+Delta|-|x|} using Taylor's formula
 for $(1+y)^{1/2}$, and on the event $A_\bx^{\rc}$ we bound \eqref{eqn:|x+Delta|-|x|} using \eqref{ass:moments}.

Indeed, for all $y > -1$, Taylor's theorem with Lagrange remainder shows that
\begin{equation*}
(1+y)^{1/2} = 1 + \frac{1}{2}y - \frac{1}{8}y^2(1+\gamma y)^{-3/2},
\end{equation*}
for some $\gamma = \gamma(y) \in [0,1]$, so on the event $A_\bx$,
\begin{align}\label{eqn:on-Ax}
\| \bx + \Delta \| - \| \bx \| &= \| \bx \| \left( \frac{\Delta_\bx}{\|\bx \|} + \frac{\| \Delta \|^2}{2\| \bx \|^2} - \frac{1}{8}\left( \frac{2 \Delta_\bx}{\|\bx\|} + \frac{\|\Delta\|^2}{\|\bx\|^2} \right)^2 \left( 1+ O(\| \bx \|^{\beta-1} ) \right) \right) \nonumber\\
&= \Delta_\bx + \frac{\| \Delta \|^2}{2\| \bx \|} - \frac{\| \bx \|}{8}\left( \frac{4 \Delta_\bx^2}{\|\bx\|^2} + \frac{\|\Delta\|^2}{\|\bx\|^2}\left( \frac{4\Delta_\bx}{\|\bx\|} + \frac{\|\Delta\|^2}{\|\bx\|^2}\right) \right) \left( 1+ O(\| \bx \|^{\beta-1} ) \right)\nonumber\\
&= \Delta_\bx + \frac{\| \Delta \|^2}{2\| \bx \|} - \frac{\Delta_\bx^2}{2\| \bx \|} \left( 1+ O(\| \bx \|^{\beta-1} ) \right) + \frac{\|\Delta\|^2}{\|\bx\|} O(\| \bx \|^{\beta-1} )\nonumber\\
&= \Delta_\bx + \left( \frac{\| \Delta \|^2 - \Delta_\bx^2 }{2\| \bx \|} \right) \left( 1+ O(\| \bx \|^{\beta-1} ) \right),
\end{align}
where the error terms follow from the fact that $| \Delta_\bx | \leq \|\Delta\| \leq \| \bx \|^\beta$ for $\beta < 1$.

On the other hand,
\begin{equation}\label{eqn:on-AxC}
\bigl| \| \bx + \Delta \| - \|\bx \| \bigr| \2 { A_\bx^{\rc} } \leq \| \Delta \| \2 { A_\bx^{\rc} } = \| \Delta\|^p \|\Delta\|^{1-p} \2 { A_\bx^{\rc} } \leq \| \Delta \|^p \| \bx \|^{\beta(1-p)},
\end{equation}
by the triangle inequality and the fact that $\| \Delta \| > \|\bx\|^\beta$ on $A_\bx^{\rc}$.
Since
\[
\| \bx + \Delta \| - \|\bx \| = ( \| \bx + \Delta \| - \|\bx \| ) \2 { A_\bx } + ( \| \bx + \Delta \| - \|\bx \| ) \2 { A_\bx^{\rc} },
\]
we can combine \eqref{eqn:on-Ax} and \eqref{eqn:on-AxC} to give
 \begin{align*}
\biggl| \| \bx + \Delta \| \biggr. &- \biggl.\|\bx \| - \left[\Delta_\bx + \left( \frac{\| \Delta \|^2 - \Delta_\bx^2 }{2\| \bx \|} \right) \left( 1+ O(\| \bx \|^{\beta-1} ) \right)  \right] \biggr|\\
& =\biggl| \| \bx + \Delta \| - \|\bx \| - \left[\Delta_\bx + \left( \frac{\| \Delta \|^2 - \Delta_\bx^2 }{2\| \bx \|} \right) \left( 1+ O(\| \bx \|^{\beta-1} ) \right)  \right] \biggr| \2 {A_\bx^{\rc} }\\
& \leq \| \Delta \|^p \| \bx \|^{\beta(1-p)} + \left| \Delta_\bx + \left( \frac{\| \Delta \|^2 - \Delta_\bx^2 }{2\| \bx \|} \right) \left( 1+ O(\| \bx \|^{\beta-1} ) \right) \right| \2 { A_\bx^{\rc} }\\
& \leq 2 \| \Delta \|^p \| \bx \|^{\beta(1-p)} + \frac{\| \Delta \|^p}{2\| \bx \|} \left( 1+ O(\| \bx \|^{\beta-1} ) \right) \|\bx\|^{\beta(2-p)} .
\end{align*}
Therefore, taking expectations and using \eqref{ass:moments}, we obtain
\[
\mu_1(\bx)
= \Exp_\bx[\Delta_\bx] + \frac{ \Exp_\bx[\| \Delta \|^2 - \Delta_\bx^2 ]}{2\| \bx \|}
+ O( \| \bx \|^{\beta-2} ) + O( \| \bx \|^{\beta(1-p)} ) + O( \| \bx \|^{\beta(2-p) - 1} ).
\]
Taking $\beta = 2/p$ makes all the error terms of size $O( \| \bx \|^{-1-\delta})$ for some $\delta = \delta(p) > 0$, namely for $\delta = (p-2)/p$.

For the second moment, we use the identity
\[
\begin{split}
(\|\bx+\Delta\| - \|\bx\|)^2 &= \|\bx +\Delta\|^2 - \|\bx \|^2 - 2\|\bx\|(\|\bx+\Delta\|-\|\bx\|)\\
&= 2\|\bx\|\Delta_\bx + \|\Delta\|^2 - 2\|\bx\|(\|\bx+\Delta\|-\|\bx\|),
\end{split}
\]
so that
\[
\mu_2(\bx)  = 2\|\bx\|\Exp_\bx[\Delta_\bx] + \Exp_\bx[\|\Delta\|^2] - 2\|\bx\|\mu_1(\bx) 
 = \Exp_\bx[\Delta_\bx^2] + O(\|\bx\|^{-\delta}),
\]
as required.
\end{proof}

With the additional assumptions \eqref{ass:zero_drift}, \eqref{ass:cov_limit}, and \eqref{ass:cov_form}, we can use Lemma~\ref{lemma1} to prove the following result.

\begin{lemma}\label{lemma2}
Suppose that $X$ is a discrete-time, time-homogeneous Markov process on $\X \subseteq \R^d$ satisfying \eqref{ass:moments},
\eqref{ass:zero_drift}, \eqref{ass:cov_limit}, and \eqref{ass:cov_form}. Then, with $\mu_1, \mu_2$ defined at \eqref{mu1}, \eqref{mu2}, and $\eps(r)$ defined at \eqref{ass:cov_limit},
there exists $\delta >0$ such that, as $\|\bx\| \to \infty$,
\begin{equation}\label{eqn:mu_k-r}
2\|\bx\| \mu_1(\bx) = V-U + O(\eps(\|\bx\|)) + O(\|\bx\|^{-\delta}), ~~ 
\mu_2(\bx) = U + O(\eps(\|\bx\|)) + O(\|\bx\|^{-\delta}).
\end{equation}
\end{lemma}
\begin{proof}
By definition of $\eps(r)$ at \eqref{ass:cov_limit} we have $\| M(\bx) - \sigma^2(\hat \bx) \|_{\rm op} = O(\eps(\|\bx\|))$ as $\|\bx\| \to \infty$.
  Then \eqref{ass:cov_form} implies that
\[
\begin{split}
\Exp_{\bx}[ \|\Delta\|^2 ] &= \trace{(M(\bx))}\\
&= \trace{(\sigma^2(\hat\bx))} + O(\eps(\|\bx\|))\\
&= V + O(\eps(\|\bx\|)),
\end{split}
\]
and
\[
\begin{split}
\Exp_{\bx}[ \Delta_\bx^2 ] &= \bra \hat\bx, M(\bx) \cdot \hat\bx\ket\\
&= \bra \hat\bx, \sigma^2(\hat\bx) \cdot \hat\bx \ket + O(\eps(\|\bx\|))\\
&= U + O(\eps(\|\bx\|)),
\end{split}
\]
and \eqref{ass:zero_drift} implies that $\Exp_{\bx} [\Delta_\bx] = \Exp_{\bx} [ \bra \Delta , \hat\bx \ket ]  = \bra \mu(\bx) , \hat\bx \ket = 0$.  
Using these expressions in Lemma~\ref{lemma1} yields \eqref{eqn:mu_k-r}.
\end{proof}

Now we can complete the proof of Theorem~\ref{thm:recurrence}.

\begin{proof}[Proof of Theorem~\ref{thm:recurrence}.]
We apply Lamperti's  \cite{lamp1} recurrence classification to $R_n = \|X_n\|$, the radial
process.  Proposition~\ref{lem:lim_sup_infty} shows
 that $\limsup_{n\to\infty}R_n = +\infty$, and Lemma~\ref{lemma1} tells us that \eqref{bound} is satisfied.

Because the error terms in \eqref{eqn:mu_k-r} are uniform in $\bx$, 
Lemma~\ref{lemma2} shows that for all $\eta > 0$ there exists $C < \infty$ such that
\[
2\|\bx\|\mu_1(\bx) - \mu_2(\bx) \in [ V-2U - \eta, V-2U + \eta ]
\]
for all $\bx \in \X$ with $\|\bx\|\geq C$.  Therefore, it follows from Theorem~3.2 of \cite{lamp1} that
$X$ is transient if $V-2U > 0$ and recurrent if $V-2U < 0$.  For the boundary case, when $V-2U=0$, if $\eps(r) = O(r^{-\delta_0})$ then
\[
2\|\bx\|\mu_1(\bx) - \mu_2(\bx) = O(\|\bx\|^{-\delta_1}),
\]
for $\delta_1 = \min\{\delta,\delta_0\}$, which implies that $X$ is recurrent, again by Theorem~3.2 of \cite{lamp1}.
\end{proof}

\section{Nullity}\label{sec:nullity}

In this section we give the proof of Theorem~\ref{thm:null}.
In the transient case, this is straightforward.

\begin{lemma}
\label{lem:null-trans}
In case (i) of Theorem~\ref{thm:recurrence}, for any bounded $A \subset \R^d$, as $n \to \infty$,
the null property \eqref{eq:null} holds.
\end{lemma}
\begin{proof}
It is sufficient to prove \eqref{eq:null} in the case where $A = B_r := \{ \bx \in \X : \| \bx \| \leq r \}$.
In case (i), $X$ is transient, meaning that $\| X_n \| \to \infty$ a.s., so that $\1 {   X_n  \in B_r } \to 0$, a.s.,
for any $r \in \RP$. Hence the Ces\`aro limit in \eqref{eq:null} is also $0$, a.s., and the  $L^q$
convergence follows from the bounded convergence theorem.
\end{proof}

It remains to consider cases (ii) and (iii), when $X$ is recurrent. Thus there exists $r_0 \in \RP$
such that $\liminf_{n \to \infty} \| X_n\| \leq r_0$, a.s.
Let $\tau_r := \min \{ n \in \ZP : X_n \in B_r \}$.
It suffices to take $A = B_r$, $r > r_0$, so $X_n \in B_r$ infinitely often.
We make the following claim, whose proof is deferred until the end of this section, which says that
if the walk has not yet entered a ball of radius $R$ (for any $R >r$ big enough), the time until it reaches
the ball of radius $r$ has tail bounded below as displayed. 

\begin{lemma}
\label{lem:null-estimate}
In cases (ii) and (iii) of Theorem~\ref{thm:recurrence}, there exists a finite $r_1 \geq
r_0$ such that for any $r>r_1$ and $R>r$ there exists a finite positive $c$ such that
\begin{equation}
\label{eq:tail_lower_bound}
 \Pr [ \tau_r  \geq n + m  \mid X_0, \ldots, X_n ] \geq c m^{-1/2} , \text{ on } \{ n < \tau_{R} \} ,
\end{equation}
for all sufficiently large $m$.
\end{lemma}

Assuming this result, we can complete the proof of Theorem~\ref{thm:null}.

\begin{proof}[Proof of Theorem~\ref{thm:null}.]
In case (i), the result is contained in Lemma~\ref{lem:null-trans}. So consider cases (ii) and (iii).
Fix $r$ and $R$ with $R > r > r_1$, with $r_1$ as in Lemma~\ref{lem:null-estimate}.  Note
that $\liminf_{n \to \infty} \| X_n\| \leq r_0 \leq r_1$, a.s.

Set $\gamma_1 := 0$ and then define recursively, for $\ell \in \N$, the stopping times
\[ \eta_\ell := \min \{ n \geq \gamma_\ell : X_{n} \notin B_{R} \} , ~~~ \gamma_{\ell+1} := \min \{ n \geq \eta_\ell : X_{n} \in B_{r}\} , \]
with the convention that $\min \emptyset := \infty$. Since $r > r_0$ and $\limsup_{n \to \infty } \| X_n \| = \infty$ (by Proposition~\ref{lem:lim_sup_infty}),
for all $\ell \in \N$ we have $\eta_\ell <\infty$ and $\gamma_\ell < \infty$, a.s., and
\[ 0 = \gamma_1 < \eta_1 < \gamma_2 < \eta_2 < \cdots  .\]
In particular,  $\lim_{\ell \to \infty} \gamma_\ell =
\lim_{\ell \to \infty} \eta_\ell = \infty$, a.s.
 
We now write $\cF_n := \sigma (X_0, \ldots, X_n)$.
We use Lemma~\ref{lem:d-dim-KOI} to show that the process must exit from $B_R$ rapidly enough.
 Indeed, 
if $\kappa$ is any finite stopping time,
set $Y_n = X_{\kappa+n} - X_{\kappa}$ and $\cG_n = \cF_{\kappa + n}$.
Then the assumptions
\eqref{ass:moments}, \eqref{ass:zero_drift} and \eqref{ass:unif_ellip} show that
 the hypotheses of  Lemma~\ref{lem:d-dim-KOI} are satisfied, since, for example,
\[ \Exp [ \| Y_{n+1} -Y_n \|^p \mid \cG_n ] = \Exp [ \| X_{\kappa+n+1} -X_{\kappa+n} \|^p \mid \cF_{\kappa+n} ] = \Exp_{X_{\kappa +n}} [ \| \Delta \|^p ] ,\]
by the strong Markov property for $X$ at the finite stopping time ${\kappa + n}$.  In particular,
another application of Lemma~\ref{lem:d-dim-KOI}, similarly to \eqref{eq:escape-chance}, shows that
we may choose $n = n(R) \in \N$ sufficiently large so that
\begin{equation}
\label{eq:big_ball_escape}
\Pr \Bigl[ \max_{0 \leq \ell \leq n(R)} \| X_{\kappa + \ell} - X_\kappa \| \geq 2 R \Bigmid \cF_{\kappa} \Bigr] \geq \frac{1}{2} , \as, \end{equation}
an event whose occurrence ensures that if $X_\kappa \in B_R$, then $X$ exits $B_R$
before time $\kappa+n(R)$.
Fix $k \in \N$. 
Then, an application of \eqref{eq:big_ball_escape} at stopping time $\kappa = \gamma_k$ shows that
\[ \Pr [ \eta_k - \gamma_k >   n (R) \mid \cF_{\gamma_k} ] 
\leq \Pr \Bigl[ \max_{0 \leq \ell \leq n(R)} \| X_{{\gamma_k} +   \ell} - X_{\gamma_k} \| < 2 R \Bigmid \cF_{\gamma_k} \Bigr] \leq \frac{1}{2} , \as\]
Similarly,
\begin{align*} \Pr [ \eta_k - \gamma_k >  2 n (R) \mid \cF_{\gamma_k} ] & = \Exp \bigl[ \1 { \eta_k - \gamma_k > n(R) } \Exp [ \1 { \eta_k - \gamma_k > 2 n(R) } \mid \cF_{\gamma_k + n(R)} ]
\bigmid \cF_{\gamma_k} \bigr] \\
& \leq \frac{1}{2}  \Pr [ \eta_k - \gamma_k >   n (R) \mid \cF_{\gamma_k} ] \leq \frac{1}{4} ,\end{align*}
this time applying \eqref{eq:big_ball_escape} at stopping time $\kappa = \gamma_k + n(R)$ as well.
Iterating this argument, it follows that $\Pr [ \eta_k - \gamma_k > m \cdot n(R) \mid \cF_{\gamma_k} ] \leq 2^{-m}$, a.s.,
 for all $m \in \N$. From here,
it is straightforward to deduce that, for some constant $C < \infty$,
for any $k \in \N$,
\begin{equation}
\label{eq:escape_upper_bound}  \Exp [  \eta_k - \gamma_k \mid \cF_{\gamma_k} ] \leq C , \as  \end{equation}
On the other hand, the tail estimate \eqref{eq:tail_lower_bound} implies that
\begin{equation}
\label{eq:return_lower_bound} \Pr [ \gamma_{k+1} -\eta_k \geq m \mid  \cF_{\eta_k} ] \geq c m^{-1/2} , \as,
 \end{equation}
for $c >0$ and all sufficiently large $m$.

For any $n \in \N$, set $k(n) := \min\{ k \geq 2 : \gamma_k > n \}$, so that
 $\gamma_{k(n)-1} \leq n < \gamma_{k(n)}$ for $k(n) \in \{2,3,\ldots\}$.
 Note $k(n) < \infty$ and $\lim_{n \to \infty} k(n) = \infty$, a.s.
Then we claim
\begin{align}
\label{eq:ratio}
 \frac{1}{n} \sum_{k=0}^{n-1} \1 { X_k \in B_{r} } 
\leq 
\frac{\sum_{k =1}^{k(n)-1} \left( \eta_{k} - \gamma_k \right)}{\sum_{k =1}^{k(n)-2} \left( \gamma_{k+1} - \eta_k \right)} .\end{align}
This is easiest to see by considering two separate cases. First, if $\eta_{k(n) -1} < n < \gamma_{k(n)}$, 
\[  \frac{1}{n} \sum_{k=0}^{n-1} \1 { X_k \in B_{r} }  \leq 
\frac{1}{\eta_{k(n)-1}}  \sum_{k=0}^{\eta_{k(n)-1}} \1 { X_k \in B_{r} }, \]
which implies \eqref{eq:ratio},
since the set of $k$ less than $n$ for which $X_k \in B_{r}$
is contained in the set $\cup_{k=1}^{k(n)-1} [\gamma_k , \eta_{k})$. On the other hand,
if $\gamma_{k(n)-1} \leq n \leq \eta_{k(n) -1}$, using the elementary inequality
$\frac{a}{b} \leq \frac{a+c}{b+c}$ for non-negative $a,b,c$ with $a/b \leq 1$, we have
\begin{align*}
  \frac{1}{n} \sum_{k=0}^{n-1} \1 { X_k \in B_{r} }  \leq 
	\frac{1}{\eta_{k(n) -1}} \left(\sum_{k=0}^{n-1} \1 { X_k \in B_{r} } + (\eta_{k(n) -1} -n) \right)	,
	\end{align*}
	which again gives \eqref{eq:ratio}.
	
To estimate the growth rates of the numerator and denominator of the right-hand side
of \eqref{eq:ratio}, we apply some results from  \cite{hmmw}.
First, writing $Z_m = \sum_{k=1}^{m-1} (\eta_k - \gamma_k)$
and $\cG_m = \cF_{\gamma_m}$, by \eqref{eq:escape_upper_bound}
we can apply Theorem 2.4 of \cite{hmmw}
to the $\cG_m$-adapted process $Z_m$ to obtain that for any $\eps>0$,
a.s., for all but finitely many $m$,
\[ \sum_{k=1}^{m   -1} \left( \eta_{k} - \gamma_k \right)  \leq m^{1+\eps} . \]
On the other hand, 
writing $Z_m = \sum_{k=1}^{m-1} (\gamma_{k+1} - \eta_k)$
and $\cG_m = \cF_{\eta_m}$, 
by \eqref{eq:return_lower_bound}
we can apply Theorem 2.6 of \cite{hmmw}
to the $\cG_m$-adapted process $Z_m$ to obtain that for any $\eps>0$,
for all $m$ sufficiently large,
\[ \sum_{k =1}^{m   -1} \left( \gamma_{k+1} - \eta_k \right) \geq m^{2-\eps} .\]
 Now \eqref{eq:ratio} gives the almost-sure version of the result \eqref{eq:null}. The $L^q$ version follows from the bounded convergence theorem.
\end{proof}

It remains to complete the proof of Lemma~\ref{lem:null-estimate}.  A more general,
two-sided version of the inequality in Lemma~\ref{lem:null-estimate} is proved
in~\cite[Theorem 2.4]{hmw1} but under slightly different assumptions.  Because of
this, we cannot apply that result directly; nevertheless, the proof techniques naturally
transfer to our setting.  In doing so, the arguments become simpler to apply, so we
reproduce them here.

\begin{proof}[Proof of Lemma~\ref{lem:null-estimate}]
By the Markov property for $X$ it is enough to prove the statement for $n=0$, namely that
there exists finite $r_1 \geq r_0$ such that for any $r > r_1$ and $R>r$ there exists a finite positive constant $c$ such that, if $X_0
\not\in B_R$ then
\[
\Pr[ \tau_r > m \mid X_0 ] \geq c m^{-1/2},
\]
for sufficiently large $m$.

We outline the two intuitive steps in the proof.
First we show that the probability that $\max_{0 \leq k \leq
  \tau_r} \|X_k\|$ exceeds some large $x$ is bounded below by
a constant times $1/x$. Second, we show that if the latter event does occur,
with  probability at least $1/2$ it takes the process time at least a constant
times $x^2$ to reach $B_r$. Combining these two estimates will show that
with probability of order $1/x$ the walk takes time of order $x^2$ to reach $B_r$, which gives the desired tail bound.
Roughly speaking, the first estimate (reaching distance $x$) is provided by the optional stopping
theorem and the fact that $\| X_k \|$ is a submartingale (cf.~\cite[Theorem 2.3]{hmw1}), and the second (taking quadratic time to return)
is provided
by a maximal inequality applied to an appropriate quadratic displacement functional (cf.~\cite[Lemma 4.11]{hmw1}). A technicality required for the first estimate
is that to apply optional stopping, we need uniform integrability; so we actually work with a truncated version of $\| X_k\|$.

We now give the details.
Recall that $R_k = \|X_k\|$ and let $\cF_k = \sigma(X_0,\dots,X_k)$.  Lemmas~\ref{lemma1}
and~\ref{lemma2}, with the fact that $V >U$ by \eqref{ass:cov_form}, imply that
\begin{equation}\label{eqn:diff-R-bound}
\Exp[ R_{k+1} - R_k \mid \cF_k ] \geq \frac{2\eps}{R_k} + o(R_k^{-1}) \geq \frac{\eps}{R_k},
\end{equation}
for all $R_k >  r_1$, for sufficiently large $r_1 \geq r_0$ and some positive constant $\eps$. 
Now, suppose that $r$ and $R$ satisfy $R> r>r_1$ and fix $x$ with $x \gg R$.  Set $R_k^x := \min\{2x,R_k\}$ and $\sigma_x := \min\{k \geq 0 : R_k > x
\}$.  Since $X_k$ is a martingale, we have that $R_k$ is a submartingale, as is the stopped
process $Y_k := R_{k \wedge \tau_r \wedge \sigma_x}$.  In order to achieve uniform
integrability, we consider the truncated process $Y_k^x := R_{k \wedge \tau_r
  \wedge \sigma_x}^x$ and show that this is a submartingale.

For $k \geq \tau_r \wedge \sigma_x$, we have $Y_{k+1}^x - Y_k^x = 0$ so $\Exp[Y_{k+1}^x
-Y_k^x \mid \cF_k ] = 0$. For $k < \tau_r \wedge \sigma_x$,
\[
Y_{k+1}^x - Y_k^x = R_{k+1} - R_k + (2x - R_{k+1}) \1{R_{k+1} > 2x},
\]
and the last term can be bounded in absolute value:
\[\begin{split}
| (2x - R_{k+1}) \1{R_{k+1} > 2x } | &\leq | R_{k+1}-R_k| \1{R_{k+1} > 2x } \\
&\leq  |R_{k+1}-R_k| \1{|R_{k+1}-R_k| > x }\\
&\leq |R_{k+1}-R_k|^p x^{1-p},
\end{split}
\]
for $p>2$ as appearing in \eqref{ass:moments}, since on $\{ k < \sigma_x\}$ we have $R_k <
  x$ and therefore $R_{k+1}> 2x$ implies that $|R_{k+1} - R_k| > x$.  Applying~\eqref{bound} from Lemma~\ref{lemma1} we obtain
\[
\Exp[ |(Y_{k+1}^x - Y_k^x) - (R_{k+1} - R_k) | \mid \cF_k ] \leq Bx^{1-p}, 
\]
for some $B<\infty$ not depending on $x$.
Combining this with \eqref{eqn:diff-R-bound} and again the fact that $R_k < x$ on $\{ k <
\sigma_x\}$, we have that
\[
\Exp[ Y_{k+1}^x - Y_k^x \mid \cF_k ] \geq \frac{\eps}{R_k} - Bx^{1-p} \geq \frac{\eps}{x}
- Bx^{1-p} \geq 0,
\]
for sufficiently large $x$.

Hence, for sufficiently large $x$, $Y_k^x$ is a uniformly integrable submartingale and
therefore, given $X_0 \not\in B_R$, by optional stopping,
\[
\begin{split}
R < R_0 = Y^x_0 \leq \Exp[ Y_{\sigma_x \wedge \tau_r}^x \mid X_0 ] &= \Exp[
Y_{\sigma_x}^x \1{\sigma_x < \tau_r} \mid X_0 ] + \Exp[ Y_{\tau_r}^x\1{\tau_r < \sigma_x}
\mid X_0 ] \\
&\leq 2x \Pr[ \sigma_x < \tau_r \mid X_0 ] + r.
\end{split}
\]
In other words, given $X_0 \not\in B_R$,
\begin{equation}\label{eqn:max-tail}
\Pr\Big[ \max_{0 \leq k \leq \tau_r} R_k > x \Bigmid X_0 \Big] \geq \frac{R-r}{2x},
\end{equation}
for all sufficiently large $x$.

Now, consider $W_k := R_{\sigma_x+k} - R_{\sigma_x}$, adapted to $\cG_k := \cF_{\sigma_x+k}$.  We have
\[
W_{k+1}^2 - W_k^2 = R_{\sigma_x+k+1}^2 - R_{\sigma_x+k}^2 - 2R_{\sigma_x}(R_{\sigma_x+k+1}-R_{\sigma_x+k}).
\]
Using the fact that $R_k$ is a submartingale together with the strong Markov property for $X$ at the
stopping time $\sigma_x+k$ yields $\Exp[ R_{\sigma_x+k+1}-R_{\sigma_x+k} \mid \cF_{\sigma_x+k} ] \geq 0 \as$,
and Lemmas~\ref{lemma1} and~\ref{lemma2} again with the strong Markov property imply that
$\Exp[ R_{\sigma_x+k+1}^2-R_{\sigma_x+k}^2 \mid \cF_{\sigma_x+k} ] \leq C \as,
$ for some constant $C<\infty$; hence $\Exp[W_{k+1}^2- W_k^2 \mid \cG_k ] \leq C \as$, for some constant $C< \infty$.  Then a maximal
inequality~\cite[Lemma~3.1]{mvw} similar to Doob's submartingale inequality implies that,
on $\{\sigma_x < \infty\}$,
\[
\Pr\Big[ \max_{0 \leq k \leq n} W_k^2 \geq y \Bigmid \cG_0 \Big] \leq \frac{Cn}{y}, \text{ for any } y >0.
\]
In particular, we may choose $\eps >0$ small enough so that
\begin{equation}\label{eqn:long-return}
\Pr\Big[ \max_{0 \leq k \leq \eps x^2} |R_{\sigma_x+k} - R_{\sigma_x}| \geq x/2 \Bigmid \cF_{\sigma_x} \Big] \leq
  \frac{1}{2}, \text{ on }   \{\sigma_x < \infty\}.
\end{equation}
Combining the inequalities~\eqref{eqn:max-tail} and~\eqref{eqn:long-return}, we find that
given $X_0 \not\in B_R$,
\[\begin{split}
\Pr\Big[\big\{\max_{0 \leq k \leq \tau_r} R_k > x \big\} &\cap \big\{ \max_{0 \leq k
  \leq \eps x^2} |R_{\sigma_x+ k} - R_{\sigma_x}| < x/2 \big\} \Bigmid X_0 \Big]\\
&= \Exp\Big[ \1{\sigma_x < \tau_r} \Pr\big[ \max_{0 \leq k \leq \eps x^2} |R_{\sigma_x+k} - R_{\sigma_x}| < x/2 \bigmid \cF_{\sigma_x} \big] \Bigmid X_0 \Big]\\
&\geq \frac{1}{2} \Pr\Big[ \max_{0 \leq k \leq \tau_r} R_k > x \Bigmid X_0 \Big] \geq \frac{R-r}{4x},
 \end{split}
\]
for sufficiently large $x$, where the equality here  uses the fact that $\{ \sigma_x < \tau_r \} \in \cF_{\sigma_x}$.

If both of the events $\{ \max_{0 \leq k \leq \tau_r} R_k > x  \}$ and
$ \{ \max_{0 \leq k \leq \eps x^2} |R_{\sigma_x+k} - R_{\sigma_x}| < x/2
 \}$ occur, then the process $X_k$ leaves the ball $B_x$ before time $\tau_r$ and takes more than
$\eps x^2$ steps to return to the ball $B_{x/2} \subset B_r$, and therefore $\tau_r > \eps
x^2$.   Setting $m = \eps x^2$ and $c = (R-r)\sqrt{\eps}/4$ yields the claimed inequality.
\end{proof}

\begin{remark}
It is only in the proof of Lemma~\ref{lem:null-estimate} that we use the condition $U < V$ from \eqref{ass:cov_form}.
In the case $U=V$, inequality \eqref{eqn:diff-R-bound} holds only for (any) $\eps <0$, and not $\eps >0$;
thus to obtain a submartingale one should look at $( Y^x_k )^\gamma$ for $\gamma >1$. The modified argument
yields a weaker version of \eqref{eq:tail_lower_bound}, with $m^{-1/2}$ replaced by $m^{-(1/2)-\delta}$ for any $\delta >0$,
but, as stated in Remark~\ref{rem:UequalsV}, this is still comfortably enough to give Theorem~\ref{thm:null} (any exponent greater than $-1$ in the tail bound will do).
We omit these additional technical details, as the case $U=V$ is outside our main interest.
\end{remark}

\appendix

\section{Recurrence in one dimension}

We use a Lyapunov function method with function $f(x) = \log ( 1 + |x| )$.

\begin{lemma}
\label{l:zero_drift_implies_recurrence}
 Suppose that $X$ is a discrete-time, time-homogeneous Markov process on $\X \subseteq \R$.
Suppose that for some $p>2$ and $v >0$,
\begin{align*} \sup_{x \in \X} \Exp [ ( X_{n+1} - X_n )^p \mid X_n = x] & < \infty ; \\
\inf_{x \in \X} \Exp [ ( X_{n+1} - X_n )^2 \mid X_n = x] & \geq v .\end{align*}
Suppose also that for some bounded set $A \subset \R$, 
\[ \Exp [   X_{n+1} - X_n   \mid X_n = x] = 0, \text{ for all } x \in \X \setminus A .\]
Then there exists a bounded set $A' \subset \R$ for which
\[\Exp [ f(X_{n+1} ) - f (X_n) \mid X_n = x ] \leq 0, \text{ for all } x \in \X \setminus A'  .\]
\end{lemma}
\begin{proof}
Write $\Delta = X_{1} - X_0$ and $E_x = \{ | \Delta | < | x | \}$.
We compute 
\begin{align*}
\Exp [ f(X_{n+1} ) - f (X_n) \mid X_n = x ] & = \Exp_x [ ( f(x + \Delta ) - f (x) ) \2 { E_x } ] \\
&{} ~~ {}  + \Exp_x [ ( f(x + \Delta ) - f (x) ) \2 { E_x^\rc } ] .
\end{align*}
On $\{  | \Delta | < | x | \}$ we have that $x$ and $x + \Delta$ have the same sign, so
\begin{align*}
& {}~~~{} \Exp_x [ ( f(x + \Delta ) - f (x) ) \2 { E_x } ] \\
& = \Exp_x \left[ \log \left( \frac{ 1 + | x + \Delta |}{1+|x|} \right) \2 { E_x }  \right]  \\
& = \Exp_x \left[ \log \left( 1 + \frac{ \Delta \sign(x)}{1+|x|} \right) \2 { E_x }  \right] \\
& \leq \left( \frac{ \sign(x)}{1+|x|} \right) \Exp_x [  \Delta \2 { E_x } ] - \frac{1}{6} ( 1+|x| )^{-2} \Exp_x [  \Delta^2 \2 { E_x } ], 
\end{align*}
using the   inequality $\log (1+ y) \leq y - \frac{1}{6} y^2$ for all $-1 < y \leq 1$.
Here, since $ \Exp_x [  \Delta ] =0$ for $x \not\in A$, 
\[ | \Exp_x [  \Delta \2 { E_x } ] |    \leq  \Exp_x [ | \Delta | \2 { E^\rc_x } ]    \leq \Exp_x [ | \Delta |^p |x|^{1-p} ] = o(|x|^{-1} ) .\]
Similarly,
\[  \Exp_x [  \Delta^2 \2 { E_x } ] \geq v - \Exp_x [  \Delta^2 \2 { E^\rc_x } ] \geq v-
o(1) .\]
(Note that here, and in what follows, our notation follows the convention as
described by \eqref{eqn:big-O}; consequently, in one dimension the error terms
are understood to be uniform as either $x \to +\infty$, or $x \to -\infty$.)
Finally we estimate the term
\begin{equation*}
\left| \Exp_x [ ( f(x + \Delta ) - f (x) ) \2 { E_x^\rc } ] \right| \leq \Exp_x \left[
  \left( \log ( 1 + | \Delta | )  + \log ( 1 + 2 |\Delta | ) \right) \2 { E_x^\rc }
\right] .
\end{equation*}
Here,
\begin{align*}
 \log ( 1 + 2 |\Delta | ) \2 { E_x^\rc } &= \log ( 1 + 2 |\Delta | )
                                           |\Delta|^p|\Delta|^{-p} \2 { E_x^\rc }  \\ 
&\leq |x|^{-p} \log ( 1 +2 |x| ) |\Delta|^p,
\end{align*}
for all $x$ with $|x|$ greater than some $x_0$ sufficiently large, using the fact that $y
\mapsto y^{-p}\log(1+2y)$ is eventually decreasing. It follows that
\begin{equation*}
\left| \Exp_x [ ( f(x + \Delta ) - f (x) ) \2 { E_x^\rc } ] \right| \leq 2|x|^{-p} \log (
1 +2 |x| )\Exp_x[|\Delta|^p] = o(|x|^{-2}).
\end{equation*}
Combining these calculations we obtain
\begin{align*}
\Exp [ f(X_{n+1} ) - f (X_n) \mid X_n = x ] & \leq \left( \frac{\sign(x)}{1 +|x|} \right) o ( |x |^{-1} ) - \frac{1}{6} (1 + |x| )^{-2} ( v - o(1) )
+ o( |x|^{-2} ) \\
& \leq - \frac{v}{6}  ( 1+|x| )^{-2} + o (|x|^{-2} ) , \end{align*}
which is negative for all $x$ with $|x|$ sufficiently large.
\end{proof}

\begin{proof}[Proof of Theorem~\ref{t:zero_drift_implies_recurrence}.]
Under assumptions \eqref{ass:moments}, \eqref{ass:zero_drift} and \eqref{ass:unif_ellip}, the
hypotheses of Lemma~\ref{l:zero_drift_implies_recurrence} are satisfied, so that for some
$x_0 \in \R_+$,
\[
\Exp [ f(X_{n+1} ) - f (X_n) \mid X_n = x ] \leq 0, \text{ for all $x \in \X$ with $|x|
  \geq x_0$},
\]
where $f(x) = \log(1+|x|)$. 

We note that assumption \eqref{ass:moments} implies that $\Exp[|X_n|] < \infty$ for all
$n$, and therefore $\Exp[f(X_n)] < \infty$ for all $n$.
Let $n_0 \in \N$ and set $\tau = \min \{n \geq n_0: |X_n| \leq x_0 \}$. Let $Y_n = f(X_{n
  \wedge \tau})$.  Then $(Y_n, n \geq n_0)$ is a non-negative supermartingale, and hence
there exists a random variable $Y_\infty \in \R_+$ with $\lim_{n\to\infty}Y_n = Y_\infty$,
a.s.  In particular, this means that 
\[ \limsup_{n\to\infty}f(X_n) \leq Y_\infty, \text{ on } \{\tau = \infty\}. \]
  Setting $\zeta = \sup \{ |x| : x \in \X, f(x) \leq Y_\infty\}$, which satisfies $\zeta<\infty$, a.s.,
 since $f(x) \to \infty$ as $|x| \to
\infty$, it follows that $\limsup_{n\to\infty}|X_n| \leq \zeta$ on $\{\tau = \infty \}$.
However, under assumptions \eqref{ass:moments}, \eqref{ass:zero_drift} and
\eqref{ass:unif_ellip},  Proposition~\ref{lem:lim_sup_infty} implies that
$\limsup_{n\to\infty}|X_n| = +\infty$, a.s., so to avoid contradiction, we must have $\tau <
\infty$, a.s.  In other words,
\[
\Pr\Big[  \inf_{n \geq n_0}|X_n| \leq x_0 \Big] = 1,
\]
and since $n_0$ was arbitrary, it follows that
\[
\Pr\Big[ \bigcap_{n_0 \in \N} \big\{\inf_{n \geq n_0}|X_n| \leq x_0\big\}\Big] = 1,
\]
which gives the result.
\end{proof}

\section*{Acknowledgements}

Part of this work was supported by the Engineering and Physical Sciences Research Council [grant number EP/J021784/1].

An antecedent of this work, concerning only the elliptic random walk in two dimensions, was written down in 2008--9 by MM and AW, who benefited from stimulating discussions with Iain MacPhee (7/11/1957--13/1/2012). The present authors also thank Stas Volkov for a comment that inspired
Remark~\ref*{rem:ellipse}\eqref{rem:Stas}.

\end{document}